\documentclass[a4paper,11pt]{amsart}

\usepackage{a4wide}
\usepackage{tikz}
\usepackage{amsmath,amsthm}
\usepackage{float}
\usepackage{enumitem}
\usepackage{amssymb}
\usepackage{cite}

\newcommand{\Pp}{\mathcal{P}}
\newcommand{\Aa}{\mathcal{A}}
\newcommand{\V}{\mathcal{V}}
\newcommand{\Ba}{\backslash}

\newtheorem{thm}{Theorem}[section]
\newtheorem{prop}[thm]{Proposition}
\newtheorem{lem}[thm]{Lemma}
\newtheorem{cor}[thm]{Corollary}
\theoremstyle{remark}
\newtheorem{rem}[thm]{Remark}
\newtheorem{defn}[thm]{Definition}
\newtheorem{exa}[thm]{Example}

\title{A polynomial associated with rooted trees and specific posets}

\author{Valisoa Razanajatovo Misanantenaina}
\address{V. Razanajatovo Misanantenaina\\
Department of Logistics\\
Stellenbosch University\\
Private Bag X1\\
Matieland 7602\\
South Africa}
\email{valisoa@sun.ac.za}
\thanks{The first author was supported by the DSI-NRF Centre of Excellence in Mathematical and Statistical Sciences (CoE-MaSS), South Africa, grant number 91486.}
\author[S. Wagner]{Stephan Wagner}
\address{S. Wagner\\
Department of Mathematics\\
Uppsala Universitet\\
Box 480\\
751 06 Uppsala\\
Sweden\\
\and
Department of Mathematical Sciences\\ 
Stellenbosch University\\
Private Bag X1\\
Matieland 7602\\
South Africa\\
}
\thanks{The second author was supported by the National Research Foundation of South Africa, grant number 96236, and the Knut and Alice Wallenberg Foundation.}
\email{stephan.wagner@math.uu.se}

\begin{document}

\begin{abstract}
We investigate a trivariate polynomial associated with rooted trees. It generalises a bivariate polynomial for rooted trees that was recently introduced by Liu. We show that this polynomial satisfies a deletion-contraction recursion and can be expressed as a sum over maximal antichains. Several combinatorial quantities can be obtained as special values, in particular the number of antichains, maximal antichains and cutsets. We prove that two of the three possible bivariate specialisations characterise trees uniquely up to isomorphism. One of these has already been established by Liu, the other is new. For the third specialisation, we construct non-isomorphic trees with the same associated polynomial.

We finally find that our polynomial can be generalised in a natural way to a family of posets that we call $\V$-posets. These posets are obtained recursively by either disjoint unions or adding a greatest/least element to existing $\V$-posets.
\end{abstract}

\maketitle

\section{Introduction and background}

There are many different polynomials that have been associated with graphs, for example the chromatic polynomial, matching polynomial, independence polynomial, characteristic polynomial, and many more. A famous instance of a bivariate polynomial associated with a graph is the \emph{Tutte polynomial}. Its basic concept can be traced back to \cite{Tut47,Tut}, in the context of colourings and flow problems in graphs. 

Tutte polynomials for trees have been investigated in \cite{chaudhary}, which introduced different versions of ranks and modified the rank formulation of the Tutte polynomial. The authors of \cite{chaudhary} showed in particular that their polynomials completely determine a rooted tree. 

Another polynomial in two variables that also has the property that two rooted trees are isomorphic if and only if they have the same associated polynomial was recently introduced by Liu in \cite{liu2021tree}: it is determined by the recursion
\[
p(T;x,y)=\begin{cases}x & \text{if $T$ has only one vertex},\\ \prod_{i=1}^k p(T_i;x,y) + y & \text{otherwise}, \end{cases}
\]
where $T_1,T_2,\ldots,T_k$ are the branches of $T$ (the rooted trees that remain as connected components when the root of $T$ is removed). Liu also gave a combinatorial interpretation of this polynomial in terms of what he calls primary subtrees. In this paper, we will extend Liu's definition to a three-variable polynomial: let $|T|$ denote the number of vertices of $T$. Set
\begin{equation}\label{eq:PRecursion}
\Pp(T;x,y,z)=\begin{cases}x & \text{if $T$ has only one vertex},\\ \prod_{i=1}^k \Pp(T_i;x,y,z)+yz^{|T|-1} & \text{otherwise}. \end{cases}
\end{equation}
Note in the second case that the exponent of $z$ can also be expressed in terms of the branches: $|T| - 1 = |T_1| + |T_2| + \cdots + |T_k|$, since the number of vertices in $T$ can be obtained by summing the number of vertices over all branches and adding $1$ for the root.
Observe that Liu's polynomial is a special case, namely $p(T;x,y) = \Pp(T;x,y,1)$. Hence it is clear that $\Pp(T;x,y,z)$ also determines $T$ uniquely. We prove that this is also the case for the two-variable specialisation $\Pp(T;1,y,z)$ (Theorem~\ref{thm:uniquely_det}), but not for $\Pp(T;x,1,z)$.

Moreover, we show that our polynomials satisfy a deletion-contraction recursion (Proposition~\ref{pro1}) and give an expansion as a sum over maximal antichains (Theorem~\ref{thmanti}). A benefit of adding the third variable is that it allows to obtain further combinatorial quantities as special cases (see Theorem~\ref{thmeval}): in particular, we obtain the number of antichains, maximal antichains and cutsets. These invariants are mainly defined for rooted trees and posets. Counting problems on antichains have been explored in \cite{Klazar} and extremal questions are considered in \cite{Andriantiana1}. On the other hand, the number of cutsets (also called transversals) in trees has been investigated in \cite{campos,gitten}, with some applications in \cite{farley}. 

In Section~\ref{secpos}, we further generalise the polynomial $\Pp(T;x,y,z)$ to a family of specific posets that we call $\V$-\emph{posets}. They can either be defined recursively, see Definition~\ref{defnvpo}, or by means of two forbidden subposets (the N-poset and the bowtie). As we will see, $\V$-posets are the natural class of posets to generalise our polynomials. They can still be expressed as a sum over maximal antichains, and yield the number of antichains and cutsets as special values.

Some of the properties of the polynomial $\Pp(T;x,y,z)$ are reminiscent of the Tutte polynomial. Let us briefly give some background for comparison. The Tutte polynomial $T(G;x,y)$  can be defined in two different ways: the first definition is recursive, in terms of deletion and contraction.

\begin{defn}
A \emph{loop} is an edge whose end vertices are the same, and a \emph{bridge} is an edge whose removal disconnects the component where it lies. Let $e$ be an edge of a graph. The \emph{deletion} $G\Ba e$ is the graph obtained from $G$ by removing $e$. The \emph{contraction} $G/e$ is the graph that results after contracting $e$ to a single vertex. Namely, the end vertices of $e$ are identified, while all other adjacencies remain the same. The Tutte polynomial $T(G;x,y)$ satisfies the following axioms:
\begin{itemize}
\item If $E(G)=\emptyset$, then $T(G;x,y)=1$.
\item If $e$ is a bridge, then $T(G;x,y)=xT(G\Ba e;x,y)$.
\item If $e$ is a loop, then $T(G;x,y)=yT(G\Ba e;x,y)$.
\item If $e$ is neither a bridge nor a loop, then
\[T(G;x,y)=T(G\Ba e;x,y)+T(G/e;x,y).\]
\end{itemize}
\end{defn}

The second definition of the Tutte polynomial focuses on its relation to the rank of graphs. We denote by $k(G)$ the number of connected components of $G$.

\begin{defn}
Let $G$ be a graph. Let $A$ be a subset of $E(G)$, and identify $A$ with the subgraph $G_{A}=(V(G),A)$. Thus all graphs $G_A$ are spanning subgraphs of $G$. We define the \emph{rank} of $A$ by 
\[r(A)=|V(G)|-k(G_A).\]
\end{defn}

\begin{rem}
We easily notice that $0 \leq r(A) \leq |A|$ with
\begin{align*}
r(A)&=0 \Longleftrightarrow A=\emptyset,\\
r(A)&=|A| \Longleftrightarrow G_A\: \text{is a forest}.
\end{align*}
\end{rem}

\begin{defn}
Let $G$ be a graph. The \emph{rank polynomial} of $G$ is defined as follows:
\[R(G;u,v)=\sum_{A \subseteq E(G)}u^{r(G)-r(A)}v^{|A|-r(A)}.\]
\end{defn}

The Tutte polynomial can now be expressed in terms of the rank polynomial.

\begin{thm}
The Tutte polynomial $T(G;x,y)$ is uniquely given by
\[T(G;x,y)=R(G;x-1,y-1).\]
\end{thm}

For more details on the Tutte polynomial and the equivalence of the two definitions, see for example~\cite[Section 9.1]{aigner} or \cite[Chapter 13]{Biggs}. The relation between the Tutte and rank polynomials leads us to the following special values.

\begin{thm}\label{tuttespecial}
For a connected graph $G$, we have
\begin{align*}
T(G;1,1)&=R(G;0,0)=\text{number of spanning trees of $G$},\\
T(G;2,1)&=R(G;1,0)=\text{number of spanning forests of $G$},\\
T(G;1,2)&=R(G;0,1)=\text{number of connected spanning subgraphs of $G$},\\
T(G;2,2)&=R(G;1,1)=2^{|E(G)|}.
\end{align*}
\end{thm}

\section{A trivariate polynomial associated with rooted trees}\label{secroot}

\subsection{Examples and a deletion-contraction algorithm}

We will now investigate the polynomial defined by the recursion~\eqref{eq:PRecursion} for rooted trees. Let us start with some simple examples.

\begin{exa}\label{ex:starpath}
For the star $S_n$ on $n$ vertices rooted at its centre, we have
\[\Pp(S_n;x,y,z)=x^{n-1}+yz^{n-1}.\]
For the path $P_n$ on $n$ vertices rooted at one of its endpoints, we have
\[\Pp(P_n;x,y,z)=x+yz+yz^2+\cdots+yz^{n-1}.\]
\end{exa}

Next we show that our polynomials can also be computed by means of a deletion-contraction recursion in analogy to the Tutte polynomial. Let us define the following operations on rooted trees.
\begin{defn}\label{defn2}
Let $e$ be an edge incident to the root and to a vertex $u_i$, and let $T_i$ be the branch attached to $u_i$. The \emph{contraction} $T/e$ is the tree that results after contracting $e$ to the root. That is, the root and $u_i$ are merged. The \emph{deletion} $T\Ba T_i$ is the tree obtained from $T$ by removing the entire branch $T_i$, including its root $u_i$. The edge between the root and $u_i$ is deleted as well.
\end{defn}

\begin{defn}
A pendant edge is an edge that is incident to a leaf. We call an edge a \emph{bridge} if it is the only edge incident to the root.
\end{defn}

\begin{prop} \label{pro1}
Let $e$ be an edge incident to the root $r$ and a vertex $u_i$. We have
\begin{enumerate}[label=(\alph*)]
\item $\Pp(T;x,y,z)=\Pp(T/e;x,y,z)+yz^{|T|-1}$ if $e$ is a bridge,
\item $\Pp(T;x,y,z)=x\Pp(T/e;x,y,z)+(z-x)yz^{|T|-2}$ if $e$ is a pendant edge,
\item $\Pp(T;x,y,z)=\Pp(T/e;x,y,z)+yz^{|T_i|-1}\Pp(T\Ba T_i;x,y,z)-y(1+y-z)z^{|T|-2}$ 
if $e$ is neither a bridge nor a pendant edge. Here, $T_i$ is the branch rooted at $u_i$.
\end{enumerate}
\end{prop}

\begin{proof}
Let $T_1,\dots,T_k$ be the branches attached to the root $r$.
\begin{enumerate}[label=(\alph*)]
\item If $e$ is a bridge attached to the branch $T_1$, then
\[\Pp(T;x,y,z)=\Pp(T_1;x,y,z)+yz^{|T|-1}=\Pp(T/e;x,y,z)+yz^{|T|-1}\]
by~\eqref{eq:PRecursion}.
\item If $e$ is a pendant edge attached to the single-vertex branch $T_1$, then, again by~\eqref{eq:PRecursion},
\begin{align*}
\Pp(T;x,y,z)-yz^{|T|-1}&=\prod_{i=1}^k \Pp(T_i;x,y,z)=\Pp(T_1;x,y,z)\prod_{i=2}^k \Pp(T_i;x,y,z)\\&=x\prod_{i=2}^k \Pp(T_i;x,y,z)=x(\Pp(T/e;x,y,z)-yz^{|T|-2})\\&=x\Pp(T/e;x,y,z)-xyz^{|T|-2},
\end{align*}
and the stated formula follows.
\item Suppose that $e$ is neither a bridge nor a pendant edge, and suppose that the ends of $e$ are the root and vertex $u_1$, the root of branch $T_1$ (without loss of generality).  Denote the branches of $T_1$ by $B_1,B_2,\dots,B_{\ell}$. We have
\begin{align*}
\Pp(T;x,y,z)-yz^{|T|-1}&=\prod_{i=1}^k \Pp(T_i;x,y,z)=\Pp(T_1;x,y,z)\prod_{i=2}^k \Pp(T_i;x,y,z)\\
&=\Big(\prod_{j=1}^{\ell}\Pp(B_j;x,y,z)+yz^{|T_1|-1}\Big)\prod_{i=2}^k \Pp(T_i;x,y,z)\\
&=\prod_{j=1}^{\ell}\Pp(B_j;x,y,z)\prod_{i=2}^k \Pp(T_i;x,y,z)+yz^{|T_1|-1}\prod_{i=2}^k \Pp(T_i;x,y,z)\\
&=\Pp(T/e;x,y,z)-yz^{|T|-2}+ yz^{|T_1|-1}\big(\Pp(T\Ba T_1;x,y,z)-yz^{|T|-|T_1|-1}\big)\\
&=\Pp(T/e;x,y,z)+yz^{|T_1|-1}\Pp(T\Ba T_1;x,y,z)-y(1+y)z^{|T|-2},
\end{align*}
which proves the third formula.
\end{enumerate}
\end{proof}

\begin{exa}\label{exatree}
For the tree in Figure~\ref{cha.treetree}, we get, using the formulas in Example~\ref{ex:starpath} and Proposition~\ref{pro1},
\begin{align*}
\Pp(T;x,y,z)&=(x^3+yz^3)+yz(x^2+yz^2)-y(1+y-z)z^3 + yz^2(x+yz+yz^2) - y(1+y-z)z^4\\
&=x^3+x^2yz+xyz^2+y^2z^3+yz^5.
\end{align*}
\begin{figure}[htbp]
\centering
\begin{tikzpicture}[scale=0.7]
\node at (-2.6,-1) {$\Pp \Bigg($};
\node[fill=black,circle,inner sep=1pt] (t1) at (-0.45,0) {};
\node[fill=black,circle,inner sep=1pt] (t2) at (-1.45,-1) {};
\node[fill=black,circle,inner sep=1pt] (t3) at (-1.95,-2) {};
\node[fill=black,circle,inner sep=1pt] (t4) at (-0.95,-2) {};
\node[fill=black,circle,inner sep=1pt] (t5) at (0.55,-1) {};
\node[fill=black,circle,inner sep=1pt] (t6) at (0.55,-2) {};
\draw (t1)--(t2)--(t3);
\draw (t2)--(t4);
\draw (t1)--(t5)--(t6);
\node at (-0.25,-1.5) {$T$};
\node at (2.7,-1) {$;x,y,z\Bigg) = \Pp \Bigg($};

\node[fill=black,circle,inner sep=1pt] (s1) at (5.6,0) {};
\node[fill=black,circle,inner sep=1pt] (s2) at (4.6,-1) {};
\node[fill=black,circle,inner sep=1pt] (s3) at (5.6,-1) {};
\node[fill=black,circle,inner sep=1pt] (s4) at (6.6,-1) {};
\node[fill=black,circle,inner sep=1pt] (s5) at (6.6,-2) {};
\draw (s1)--(s4)-- (s5);
\draw (s1)--(s3);
\draw (s1)--(s2);
\node at (9.1,-1) {$;x,y,z\Bigg) + yz^2 \Pp \Bigg($};

\node[fill=black,circle,inner sep=1pt] (u1) at (11.7,0) {};
\node[fill=black,circle,inner sep=1pt] (u2) at (11.7,-1) {};
\node[fill=black,circle,inner sep=1pt] (u3) at (11.7,-2) {};
\draw (u1)--(u2)--(u3);
\node at (15.2,-1) {$;x,y,z\Bigg) - y(1+y-z)z^4$};

\node at (-2.6,-4) {$\Pp \Bigg($};
\node[fill=black,circle,inner sep=1pt] (s01) at (-1.15,-3) {};
\node[fill=black,circle,inner sep=1pt] (s02) at (-2.15,-4) {};
\node[fill=black,circle,inner sep=1pt] (s03) at (-1.15,-4) {};
\node[fill=black,circle,inner sep=1pt] (s04) at (-0.15,-4) {};
\node[fill=black,circle,inner sep=1pt] (s05) at (-0.15,-5) {};
\draw (s01)--(s04)-- (s05);
\draw (s01)--(s03);
\draw (s01)--(s02);
\node at (2,-4) {$;x,y,z\Bigg) = \Pp \Bigg($};

\node[fill=black,circle,inner sep=1pt] (v1) at (5,-3.5) {};
\node[fill=black,circle,inner sep=1pt] (v2) at (4,-4.5) {};
\node[fill=black,circle,inner sep=1pt] (v3) at (5,-4.5) {};
\node[fill=black,circle,inner sep=1pt] (v4) at (6,-4.5) {};
\draw (v1)--(v2);
\draw (v1)--(v3);
\draw (v1)--(v4);
\node at (8.35,-4) {$;x,y,z\Bigg) + yz \Pp \Bigg($};
\node[fill=black,circle,inner sep=1pt] (w1) at (11.1,-3.5) {};
\node[fill=black,circle,inner sep=1pt] (w2) at (10.6,-4.5) {};
\node[fill=black,circle,inner sep=1pt] (w3) at (11.6,-4.5) {};
\draw (w1)--(w2);
\draw (w1)--(w3);
\node at (15,-4) {$;x,y,z\Bigg) - y(1+y-z)z^3$};
\end{tikzpicture}
\caption{Computing $\Pp(T;x,y,z)$ using the deletion-contraction recursion.}
\label{cha.treetree}
\end{figure}
\end{exa}

\subsection{The polynomial $\Pp(T;x,y,z)$ as a generating function}

Next we relate the polynomial $\Pp(T;x,y,z)$ to the structure of the tree by showing that it can be regarded as a generating function for maximal antichains capturing different properties of these antichains.

\begin{defn}
An \emph{antichain} in a rooted tree is a set of vertices with the property that no two vertices in the set lie on a common path from the root. A \emph{maximal antichain} is an antichain that is not a proper subset of another antichain.
\end{defn}

Let $A$ be a maximal antichain in $T$. We write $\ell(A)$ for the number of leaves in $A$.
It is easy to see that $0 \leq \ell(A) \leq |A|$ and
\begin{align*}
\ell(A)=0 &\Longleftrightarrow A \; \text{does not contain any leaf},\\
\ell(A)=|A| &\Longleftrightarrow A \; \text{is the set of all leaves of $T$}.
\end{align*}

\begin{defn}
For a vertex $v$ of a rooted tree $T$, we let $T_v$ be the subtree formed by $v$ and all its successors. Such a tree is called a \emph{fringe subtree} of $T$. Note that the branches of a rooted tree are precisely the fringe subtrees associated with the root's children.
\end{defn}

For a maximal antichain $A$ of $T$, we define $s(A)$ as the number of vertices below the antichain $A$ (total number of successors of vertices in $A$), which can be expressed as
\[s(A)=\sum_{a \in A} (|T_a|-1).\]
Again, one easily finds that $0 \leq s(A) \leq |T|-1$ and
\begin{align*}
s(A)=0 &\Longleftrightarrow A \; \text{is the set of all leaves of $T$},\\
s(A)=|T|-1 &\Longleftrightarrow A \; \text{only consists of the root}.
\end{align*}

Let $\Aa(T)$ be the set of all maximal antichains of $T$.

\begin{thm}\label{thmanti}
For every rooted tree $T$,
\[\Pp(T;x,y,z)=\sum_{A \in \Aa(T)}x^{\ell(A)}y^{|A|-\ell(A)}\,z^{s(A)}.\]
\end{thm}

\begin{proof}
Let us denote the right side of the equation by $R(T;x,y,z)$, i.e.,
\[R(T;x,y,z) = \sum_{A \in \Aa(T)}x^{\ell(A)}y^{|A|-\ell(A)}\,z^{s(A)}.\]
It is straightforward that $R(\bullet;x,y,z)=x=\Pp(\bullet;x,y,z)$, since the only maximal antichain of the single-vertex tree $\bullet$ consists of the root.

Now for $T \neq \bullet$, we have to verify that $R(T;x,y,z)$ satisfies~\eqref{eq:PRecursion}. Let $T_1,\dots,T_k$ be the branches attached to the root. Every maximal antichain of $T$ is either a union of maximal antichains in the branches or consists only of the root:
\[\Aa(T)=\{A_1\cup A_2\cup \dots \cup A_k \,:\, A_i \in \Aa(T_i) \text{ for } 1\leq i \leq k\}\cup\{\{r\}\}.\]
Note that the number of leaves $\ell(A)$ in a maximal antichain $A$ is the sum of the number of leaves in its parts $A_i$:
\[\ell(A) = \ell(A_1) + \ell(A_2) + \cdots + \ell(A_k).\]
The same applies to the number of successors $s(A)$. Thus,
\begin{align*}
R(T;x,y,z)&=\sum_{A \in \Aa(T)}x^{\ell(A)}\,y^{|A|-\ell(A)}z^{s(A)}\\
&= \sum_{A_1 \in \Aa(T_1)} \sum_{A_2 \in \Aa(T_2)} \cdots \sum_{A_k \in \Aa(T_k)} x^{\ell(A_1) + \cdots + \ell(A_k)} y^{|A_1| - \ell(A_1) + \cdots + |A_k|  - \ell(A_k)} z^{s(A_1) + \cdots + s(A_k)} \\
&\quad +x^{\ell(\{r\})}\,y^{|\{r\}| - \ell(\{r\})}z^{s(\{r\})}\\
&=\prod_{i=1}^k\Big(\sum_{A_i \in \Aa(T_i)}x^{\ell(A_i)}\,y^{|A_i|-\ell(A_i)}z^{s(A_i)}\Big)+x^0\,y^{1-0}z^{|T|-1}\\
&=\prod_{i=1}^k R(T_i;x,y,z)+yz^{|T|-1}.
\end{align*} 
\end{proof}

Next we show that several special values of the polynomial $\Pp(T;x,y,z)$ have a natural combinatorial interpretation. Let us first consider the special values that result when one of the variables is equal to $0$:

\begin{cor}\label{cor:yz0}
We have
\[\Pp(T;x,y,0) = \Pp(T;x,0,z) = x^{\text{number of leaves in $T$}}.\]
\end{cor}

\begin{proof}
We use the representation of $\Pp(T;x,y,z)$ in Theorem~\ref{thmanti}. Note that all contribution to the sum that is nonzero when $y = 0$ or $z=0$ stems from the maximal antichain that consists of all leaves. The statement follows immediately.
\end{proof}

\begin{cor}\label{cor:x0}
Let $\Aa^*(T)$ be the set of maximal antichains of $T$ that do not contain a leaf. We have
\[\Pp(T;0,y,z) = \sum_{A \in \Aa^*(T)}y^{|A|}\,z^{s(A)}.\]
In particular, $\Pp(T;0,1,1)$ is the number of maximal antichains of $T$ that do not contain a leaf.
\end{cor}

\begin{proof}
This is also immediate from Theorem~\ref{thmanti}. 
\end{proof}

Next, we have a number of special values reminiscent of Theorem~\ref{tuttespecial} for the Tutte polynomial.
Let us define a \emph{cutset} (or transversal) in a rooted tree as a set of vertices that meets every path from the root to a leaf.

\begin{thm}\label{thmeval}
We have the following special values:
\begin{align*}
\Pp(T;1,1,1) &= \text{number of maximal antichains of $T$},\\
\Pp(T;2,1,1) &=\text{number of antichains of $T$ (including $\emptyset$)},\\
\Pp(T;1,1,2) &=\text{number of cutsets of $T$},\\
\Pp(T;2,1,2) &=2^{|T|}.
\end{align*}
\end{thm}

\begin{proof}
The first item is straightforward from Theorem~\ref{thmanti}.

The second item comes from the fact that every antichain can be obtained uniquely from a maximal antichain by removing a subset (possibly empty) of the leaves contained in it. Alternatively, one can also use the recursion, noting that an antichain is either a union of antichains in the branches, or consists only of the root.

The third formula is due to the fact that a cutset can be formed uniquely from a maximal antichain by including a subset (possibly empty) of the successors. From a recursive perspective, note that a cutset can either be any set that contains the root (of which there are $2^{|T|-1}$), or a union of cutsets in the branches. Hence, the number of cutsets satisfies recursion~\eqref{eq:PRecursion} with $x=y=1$ and $z=2$.

The last item is easily proven by induction: if holds when $T$ is a single vertex. If $T$ has branches $T_1,T_2,\ldots,T_k$, then we get from \eqref{eq:PRecursion} and the induction hypothesis that
\[\Pp(T;2,1,2) = \prod_{i=1}^k \Pp(T_i;2,1,2) + 2^{|T|-1} =2^{\sum_{i=1}^k |T_i|} + 2^{|T|-1} = 2^{|T|-1} + 2^{|T|-1} = 2^{|T|}.\]
\end{proof}

\begin{rem}\label{remideal}
For every subtree $S$ of $T$ that contains the root, the leaves of $S$ form an antichain. Conversely, given an antichain, the union of the paths from the root is a subtree that contains the root. Thus $\Pp(T;2,1,1)$ is also the number of subtrees containing the root.
\end{rem}

\begin{rem}
The final identity of Theorem~\ref{thmeval} can in fact be generalised in a straightforward fashion to
\[\Pp(T;1+y,y,1+y) =(1+y)^{|T|}.\]
\end{rem}

\subsection{Independent sets and matchings}

Our next result relates the polynomial $\Pp(T;x,y,z)$ to maximum matchings and maximum independent sets. Recall that a \emph{matching} is a set of edges with the property that no two of them share a vertex. An \emph{independent set} is a set of vertices with the property that no two of them are adjacent. The greatest possible cardinality of a matching in a graph $G$ is called the \emph{matching number} and denoted by $\mu(G)$; a matching of that cardinality is also called a \emph{maximum matching}. Likewise, the maximum cardinality of an independent set is the \emph{independence number} $\alpha(G)$, and an independent set of that cardinality is called a \emph{maximum independent set}.

\begin{thm}
Let $T$ be a rooted tree. The following three statements are equivalent:
\begin{enumerate}[label=(\alph*)]
\item Every maximum independent set of $T$ contains the root.
\item There exists a maximum matching of $T$ that does not cover the root (i.e., does not contain an edge that is incident with the root).
\item $\Pp(T;0,-1,-1) = 0$.
\end{enumerate}
If the three statements do not hold, then $\Pp(T;0,-1,-1) = (-1)^{|T|}$.
\end{thm}

\begin{proof}
We first show that the first two items are equivalent: let $r$ be the root. By König's theorem \cite[Theorem 2.1.1]{diestel}, $\mu(T)$ is also the minimum cardinality of a vertex cover, and since vertex covers are precisely the complements of independent sets, $\alpha(T) + \mu(T) = |T|$. Likewise, $\alpha(T \Ba r) + \mu(T \Ba r) = |T \Ba r|  = |T| - 1$.

Clearly, $\alpha(T \Ba r) \leq \alpha(T) \leq \alpha(T \Ba r) + 1$. If there is a maximum independent set of $T$ that does not contain the root, then $\alpha(T \Ba r) = \alpha(T)$, and vice versa. So statement (a) is equivalent to $\alpha(T) = \alpha(T \Ba r) + 1$, which is in turn equivalent to $\mu(T) = \mu(T \Ba r)$.

Similarly, we have $\mu(T \Ba r) \leq \mu(T) \leq \mu(T \Ba r) + 1$. If there is a maximum matching of $T$ that does not cover the root, then $\mu(T \Ba r) = \mu(T)$, and vice versa. So we have shown that (a) and (b) are equivalent.

Now let us call a rooted tree a \emph{type A} tree if it satisfies (a) and (b); otherwise we say it is of \emph{type B}. It was shown in \cite[Lemma 3.3]{lehner} that $T$ is of type A if and only if all its branches are of type B. We now prove by induction on the number of vertices that $T$ is of type A if and only if $\Pp(T;0,-1,-1) = 0$, and of type B if and only if $\Pp(T;0,-1,-1) = (-1)^{|T|}$. This holds for the single-vertex tree, which is of type A and satisfies $\Pp(\bullet;0,-1,-1) = 0$. For the induction step, note that we have, by~\eqref{eq:PRecursion},
\[\Pp(T;0,-1,-1) = \prod_{i=1}^k \Pp(T_i;0,-1,-1) + (-1)^{|T|}.\]
The product is $0$ if at least one of the factors is, or equivalently, if at least one branch $T_i$ is of type A. In this case, $T$ is of type B and $\Pp(T;0,-1,-1) = (-1)^{|T|}$. Otherwise, $T$ is of type A and
\[\Pp(T;0,-1,-1) = \prod_{i=1}^k (-1)^{|T_i|} + (-1)^{|T|} = (-1)^{|T|-1} + (-1)^{|T|} = 0.\]
This completes the induction.
\end{proof}

Note that
\begin{align*}
\mu(T) &= \mu(T \Ba r) + \begin{cases} 0 & T \text{ is of type A,} \\ 1 & T \text{ is of type B,} \end{cases} \\ 
&= \sum_{i=1}^k \mu(T_i) + (-1)^{|T|} \Pp(T;0,-1,-1).
\end{align*}

By induction on the number of vertices, the following formula for the matching number now follows easily.

\begin{cor}
For every rooted tree $T$, we have
\[\mu(T) = \sum_{v \in T} (-1)^{|T_v|} \Pp(T_v;0,-1,-1).\]
The sum is over all fringe subtrees $T_v$ of $T$.
\end{cor}

\subsection{Random removal of vertices or edges}

As a final result on special values of $\Pp(T;x,y,z)$, we consider a simple probabilistic question: suppose that each vertex of the tree $T$ is removed with probability $p$. What is the probability that at least one complete path from the root to a leaf of $T$ remains? The answer is a polynomial in $p$, similar to the \emph{reliability polynomial} \cite{chari}, which can be expressed in terms of the Tutte polynomial. As we will see, the probability can be expressed in terms of $\Pp(T;x,y,z)$.

\begin{prop}\label{prop:probability1}
Let $p \in (0,1)$ be a fixed probability. If each vertex of the tree $T$ is removed independently with probability $p$, the probability that no complete path from the root to a leaf remains is
\[(1-p)^{|T|} \Pp \Big(T; \frac{p}{1-p}, \frac{p}{1-p}, \frac{1}{1-p} \Big).\]
\end{prop}

\begin{proof}
Let the probability in question be denoted by $Q(T)$. If the root is removed, then clearly there is no complete path from the root to a leaf. Otherwise, there is no such path if and only if there is none in any of the branches. Hence, we have
\[Q(T) = p + (1-p) \prod_{i=1}^k Q(T_i).\]
Divide both sides by $(1-p)^{|T|}$ and use the fact that $\prod_{i=1}^k (1-p)^{|T_i|} =(1-p)^{\sum_{i=1}^k |T_i|} = (1-p)^{|T|-1}$:
\[Q(T) (1-p)^{-|T|} = p (1-p)^{-|T|} + \prod_{i=1}^k Q(T_i) (1-p)^{-|T_i|}.\]
Thus $Q(T) (1-p)^{-|T|}$ satisfies the recursion~\eqref{eq:PRecursion} with $y = p/(1-p)$, $z = 1/(1-p)$ and initial value $x = Q(\bullet) (1-p)^{-1} = p/(1-p)$. This proves the statement.
\end{proof}

One can also consider the following dual question.

\begin{prop}
Let $p \in (0,1)$ be a fixed probability. If each edge of the tree $T$ is removed with probability $p$, the probability that no complete path from the root to a leaf remains is
\[(1-p)^{|T|-1} \Pp \Big(T; \frac{p}{1-p}, \frac{p}{1-p}, \frac{1}{1-p} \Big) - \frac{p}{1-p}.\]
\end{prop}

Motivated by this type of probabilistic question, Burghart \cite{burghart} also considered a tree polynomial that can be expressed as $1-x^{|T|} \Pp(T; \frac{1-x}{x}, \frac{1-x}{x}, \frac{1}{x})$ and has interesting applications of its own right.

\medskip

It is natural to ask to what extent a polynomial associated with a tree determines the tree. As mentioned before, the main result of \cite{liu2021tree} states that $T$ is uniquely determined by the bivariate polynomial $\Pp(T;x,y,1)$. Thus, the trivariate polynomial $\Pp(T;x,y,z)$ \emph{a fortiori} determines $T$ uniquely.

\begin{figure}[htbp]
\centering
\begin{tikzpicture}[scale=0.7]
\node[fill=black,circle,inner sep=1pt] (t1) at (0,0) {};
\node[fill=black,circle,inner sep=1pt] (t2) at (0,-1) {};
\node[fill=black,circle,inner sep=1pt] (t3) at (-1,-2) {};
\node[fill=black,circle,inner sep=1pt] (t4) at (0,-2) {};
\node[fill=black,circle,inner sep=1pt] (t5) at (1,-2) {};
\node[fill=black,circle,inner sep=1pt] (t6) at (-1,-3) {};
\node[fill=black,circle,inner sep=1pt] (t7) at (0,-3) {};
\node[fill=black,circle,inner sep=1pt] (t8) at (1,-3) {};
\draw (t1)--(t2);
\draw (t2)--(t3)--(t6);
\draw (t2)--(t4)--(t7);
\draw (t2)--(t5)--(t8);
\node at (0,-3.5) {$T_1$};
\node[fill=black,circle,inner sep=1pt] (s1) at (3,0) {};
\node[fill=black,circle,inner sep=1pt] (s2) at (2,-1) {};
\node[fill=black,circle,inner sep=1pt] (s3) at (4,-1) {};
\node[fill=black,circle,inner sep=1pt] (s4) at (2,-2) {};
\node[fill=black,circle,inner sep=1pt] (s5) at (2,-3) {};
\node[fill=black,circle,inner sep=1pt] (s6) at (3.5,-2) {};
\node[fill=black,circle,inner sep=1pt] (s7) at (3.5,-3) {};
\node[fill=black,circle,inner sep=1pt] (s8) at (4.5,-2) {};
\draw (s1)--(s2)--(s4)-- (s5);
\draw (s1)--(s3)--(s8);
\draw (s3)--(s6)--(s7);
\node at (3,-3.5) {$T_2$};
\node[fill=black,circle,inner sep=1pt] (x1) at (6,0) {};
\node[fill=black,circle,inner sep=1pt] (x2) at (6,-1) {};
\node[fill=black,circle,inner sep=1pt] (x3) at (5,-2) {};
\node[fill=black,circle,inner sep=1pt] (x4) at (7,-2) {};
\draw (x1)--(x2)--(x3);
\draw (x2)--(x4);
\node at (6,-2.5) {$T_3$};
\node[fill=black,circle,inner sep=1pt] (y1) at (9,0) {};
\node[fill=black,circle,inner sep=1pt] (y2) at (8,-1) {};
\node[fill=black,circle,inner sep=1pt] (y3) at (8,-2) {};
\node[fill=black,circle,inner sep=1pt] (y4) at (10,-1) {};
\draw (y3)--(y2)--(y1)--(y4);
\node at (9,-2.5) {$T_4$};
\node[fill=black,circle,inner sep=1pt] (u1) at (12,0) {};
\node[fill=black,circle,inner sep=1pt] (u2) at (11,-1) {};
\node[fill=black,circle,inner sep=1pt] (u3) at (13,-1) {};
\node[fill=black,circle,inner sep=1pt] (u4) at (11,-2) {};
\node[fill=black,circle,inner sep=1pt] (u5) at (11,-3) {};
\node[fill=black,circle,inner sep=1pt] (u6) at (11,-4) {};
\draw (u1)--(u2)--(u4)-- (u5)--(u6);
\draw (u1)--(u3);
\node at (12,-4.5) {$T_5$};
\node[fill=black,circle,inner sep=1pt] (v1) at (15,0) {};
\node[fill=black,circle,inner sep=1pt] (v2) at (14,-1) {};
\node[fill=black,circle,inner sep=1pt] (v3) at (16,-1) {};
\node[fill=black,circle,inner sep=1pt] (v4) at (13.5,-2) {};
\node[fill=black,circle,inner sep=1pt] (v5) at (14.5,-2) {};
\node[fill=black,circle,inner sep=1pt] (v6) at (16,-2) {};
\draw (v1)--(v2)--(v4);
\draw (v1)--(v3)--(v6);
\draw (v2)--(v5);
\node at (15,-2.5) {$T_6$};
\end{tikzpicture}
\caption{Pairs of nonisomorphic rooted trees.}
\label{cha.iso}
\end{figure}

However, none of the single-variable polynomials $\Pp(x,1,1)$, $\Pp(1,y,1)$ and $\Pp(1,1,z)$ suffices to distinguish trees uniquely, as the following examples show. In each case, they are the smallest possible examples, as a computer search exhibits.
Let us consider the trees in Figure~\ref{cha.iso}. The polynomials evaluate as follows:

\begin{align*}
\Pp(T_1;x,1,1)&=\Pp(T_2;x,1,1)=x^3+3x^2+3x+3,\\
\Pp(T_3;1,y,1)&=\Pp(T_4;1,y,1)=2y+1,\\
\Pp(T_5;1,1,z)&=\Pp(T_6;1,1,z)=z^5+z^3+z^2+z+1.
\end{align*}

However, note that
\begin{align*}
\Pp(T_1;x,y,z)&= yz^7+yz^6+y^3z^3+3xy^2z^2+3x^2yz+x^3,\\
\Pp(T_2;x,y,z)&= yz^7+y^2z^5+y^2z^4+xy^2z^3 + xyz^3 + xy^2z^2 + x^2yz^2 + 2x^2yz + x^3,\\
\Pp(T_3;x,y,z)&= yz^3+yz^2+x^2,\\
\Pp(T_4;x,y,z)&= yz^3+xyz+x^2,\\
\Pp(T_5;x,y,z)&= yz^5+xyz^3 + xy z^2 + x yz + x^2,\\
\Pp(T_6;x,y,z)&= yz^5+y^2z^3+xyz^2+x^2yz+x^3,
\end{align*}
so we have $\Pp(T_1;x,y,z) \neq \Pp(T_2;x,y,z)$, $\Pp(T_3;x,y,z) \neq \Pp(T_4;x,y,z)$ and $\Pp(T_5;x,y,z) \neq \Pp(T_6;x,y,z)$.
This leaves us with bivariate specialisations of $\Pp(T;x,y,z)$. We will show next that the bivariate polynomial $\Pp(T;1,y,z)$ determines $T$ uniquely. First we prove irreducibility of this polynomial, which is also the key step in \cite{liu2021tree} to prove the analogous statement for $\Pp(T;x,y,1)$:

\begin{lem}
For every rooted tree $T$ with more than one vertex, the polynomial $\Pp(T;1,y,z)$ is irreducible.
\end{lem}

\begin{proof}
Let $T$ have $n$ vertices. We first observe that the degree of $\Pp(T;1,y,z)$ as a polynomial in $z$ is always $n-1$, and that the term of highest degree is $yz^{n-1}$ if $n \geq 2$ (and $1$ if $n=1$). This is clear from Theorem~\ref{thmanti}: $s(A)$ is the number of vertices below $A$, which is clearly at most $n-1$, with equality if and only if $A$ consists only of the root. Next note that
$\Pp(T;1,0,z) = 1$ by Corollary~\ref{cor:yz0}, which means that all but the constant term $1$ are divisible by $y$. Thus Eisenstein's criterion \cite[Chapter IV, Theorem 3.1]{lang} applies to the ``reversed'' polynomial $z^{n-1} \Pp(T;1,y,1/z)$: its degree (regarded as a polynomial in $z$) is still $n-1$, the leading term is $z^{n-1}$, all other terms are divisible by $y$, but the coefficient of $z^0$ is $y$ and thus not divisible by $y^2$. So $z^{n-1} \Pp(T;1,y,1/z)$, and thus also the original polynomial $\Pp(T;1,y,z)$, is irreducible.
\end{proof}

Now the following statement is an easy consequence:

\begin{thm}\label{thm:uniquely_det}
The bivariate polynomial $\Pp(T;1,y,z)$ determines $T$ uniquely: if we have $\Pp(T;1,y,z) = \Pp(T';1,y,z)$, then $T$ and $T'$ are isomorphic.
\end{thm}

\begin{proof}
Note first that $T$ and $T'$ have the same number of vertices $n$: otherwise, the two polynomials would have different degrees (in $z$). We proceed by induction on $n$. The initial case $n=1$ is trivial.

Now let $T_1,T_2,\ldots,T_k$ and $T_1',T_2',\ldots,T_{\ell}'$ be the branches of $T$ and $T'$ respectively. It follows from the recursion~\eqref{eq:PRecursion} now that
\begin{equation}\label{eq:productsequal}
\prod_{i=1}^k \Pp(T_i;1,y,z) =  \prod_{i=1}^{\ell} \Pp(T_j';1,y,z).
\end{equation}
Since the factors on both sides are irreducible polynomials, they must be the same up to permutation, with one exception: factors where $T_i$ ($T_j'$, respectively) is a single vertex are equal to $1$, and their number is a priori not uniquely determined by~\eqref{eq:productsequal}. However, by the induction hypothesis, we can infer that the branches of $T$ and $T'$ that have more than one vertex are the same (up to isomorphism). Since $T$ and $T'$ also have the same number of vertices, the number of single-vertex branches must also be the same. Hence $T$ and $T'$ are isomorphic, completing the proof.
\end{proof}

Interestingly, the third possible two-variable specialisation $\Pp(T;x,1,z)$ does \emph{not} uniquely determine trees. We will exhibit this fact by means of a counterexample (the smallest counterexample, in fact, as can be verified by a computer search). 

\begin{figure}[htbp]
\centering
\begin{tikzpicture}[scale=0.7]
\node[fill=black,circle,inner sep=1.5pt]  at (0,0) {};
\node[fill=black,circle,inner sep=1.5pt]  at (-1.5,-1) {};
\node[fill=black,circle,inner sep=1.5pt]  at (0,-1) {};
\node[fill=black,circle,inner sep=1.5pt]  at (1.5,-1) {};
\node[fill=black,circle,inner sep=1.5pt]  at (-1.5,-2) {};
\node[fill=black,circle,inner sep=1.5pt]  at (-0.5,-2) {};
\node[fill=black,circle,inner sep=1.5pt]  at (0.5,-2) {};
\node[fill=black,circle,inner sep=1.5pt]  at (-1.5,-3) {};
\node[fill=black,circle,inner sep=1.5pt]  at (-0.5,-3) {};
\node[fill=black,circle,inner sep=1.5pt]  at (-2.5,-4) {};
\node[fill=black,circle,inner sep=1.5pt]  at (-1.5,-4) {};
\node[fill=black,circle,inner sep=1.5pt]  at (-0.5,-4) {};
\node[fill=black,circle,inner sep=1.5pt]  at (-2.5,-5) {};
\node[fill=black,circle,inner sep=1.5pt]  at (-1.5,-5) {};

\draw (0,0)--(-1.5,-1);
\draw (0,0)--(0,-1);
\draw (0,0)--(1.5,-1);
\draw (-1.5,-1)--(-1.5,-3);
\draw (0,-1)--(-0.5,-2);
\draw (0,-1)--(0.5,-2);
\draw (-0.5,-2)--(-0.5,-3);
\draw (-1.5,-3)--(-2.5,-4);
\draw (-2.5,-4)--(-2.5,-5);
\draw (-1.5,-3)--(-1.5,-5);
\draw (-1.5,-3)--(-0.5,-4);

\node at (0,-5) {$B_1$};

\node[fill=black,circle,inner sep=1.5pt]  at (4,0) {};
\node[fill=black,circle,inner sep=1.5pt]  at (3.5,-1) {};
\node[fill=black,circle,inner sep=1.5pt]  at (5,-1) {};
\node[fill=black,circle,inner sep=1.5pt]  at (3,-2) {};
\node[fill=black,circle,inner sep=1.5pt]  at (4.5,-2) {};
\node[fill=black,circle,inner sep=1.5pt]  at (2.5,-3) {};
\node[fill=black,circle,inner sep=1.5pt]  at (3.5,-3) {};
\node[fill=black,circle,inner sep=1.5pt]  at (4.5,-3) {};
\node[fill=black,circle,inner sep=1.5pt]  at (2.5,-4) {};
\node[fill=black,circle,inner sep=1.5pt]  at (4.5,-4) {};

\draw (4,0)--(5,-1);
\draw (4,0)--(2.5,-3);
\draw (2.5,-3)--(2.5,-4);
\draw (3,-2)--(3.5,-3);
\draw (3.5,-1)--(4.5,-2);
\draw (4.5,-2)--(4.5,-4);

\node at (4,-5) {$B_2$};

\node[fill=black,circle,inner sep=1.5pt]  at (7,0) {};
\node[fill=black,circle,inner sep=1.5pt]  at (7,-1) {};
\node[fill=black,circle,inner sep=1.5pt]  at (6,-2) {};
\node[fill=black,circle,inner sep=1.5pt]  at (7,-2) {};
\node[fill=black,circle,inner sep=1.5pt]  at (8,-2) {};
\node[fill=black,circle,inner sep=1.5pt]  at (6,-3) {};
\node[fill=black,circle,inner sep=1.5pt]  at (7,-3) {};

\draw (7,0)--(7,-3);
\draw (7,-1)--(6,-2);
\draw (7,-1)--(8,-2);
\draw (6,-2)--(6,-3);

\node at (7,-4) {$B_3$};

\node[fill=black,circle,inner sep=1.5pt]  at (9,0) {};
\node[fill=black,circle,inner sep=1.5pt]  at (9,-1) {};
\node[fill=black,circle,inner sep=1.5pt]  at (9,-2) {};

\draw (9,0)--(9,-2);

\node at (9,-3) {$B_4$};

\end{tikzpicture}
\caption{Trees $B_1,B_2,B_3,B_4$.}\label{fig:4trees}
\end{figure}

Consider the trees $B_1,B_2,B_3,B_4$ in Figure~\ref{fig:4trees}. One verifies that
\begin{align*}
\Pp(B_1;x,1,z) &= (x^2-xz^2+xz+z^4) (x^4+x^3z^2+2x^3z+2x^2z^3+x^2z^2+xz^7+xz^5+xz^4+z^9) \\
&= x^6 + 3 x^5 z + x^4 z^3 + 3 x^4 z^2 + x^3 z^7 + x^3 z^6 + x^3 z^5 + 2 x^3 z^4 + x^3 z^3 + x^2 z^8 + x^2 z^7 \\
&\quad + x^2 z^6 + x^2 z^5 + x z^{10} + x z^9 + x z^8 + z^{13}, \\
\Pp(B_2;x,1,z) &= x^4+x^3z^2+2x^3z+2x^2z^3+x^2z^2+xz^7+xz^5+xz^4+z^9, \\
\Pp(B_3;x,1,z) &= (x+z^2+z)(x^2-xz^2+xz+z^4) \\
&= x^3 + 2 x^2 z + x z^2 + z^6 + z^5, \\
\Pp(B_4;x,1,z) &= x+z^2+z.
\end{align*}

Note that $\Pp(B_1;x,1,z) \Pp(B_4;x,1,z) = \Pp(B_2;x,1,z) \Pp(B_3;x,1,z)$. It follows that the two rooted $18$-vertex trees $A_1$ and $A_2$ whose branches are $B_1,B_4$ and $B_2,B_3$ respectively (see Figure~\ref{fig:2trees}) have the same associated polynomial:
\begin{align*}
\Pp(A_1;x,1,z) &= \Pp(A_2;x,1,z) \\
&= \Pp(B_1;x,1,z) \Pp(B_4;x,1,z) + z^{17} = \Pp(B_2;x,1,z) \Pp(B_3;x,1,z) + z^{17} \\
&= x^7+x^6 z^2+4 x^6 z+4 x^5 z^3+6 x^5 z^2+x^4 z^7+x^4 z^6+2 x^4 z^5+6 x^4 z^4+4 x^4
   z^3 \\ 
&\quad + x^3 z^9 +3 x^3 z^8+3 x^3 z^7+4 x^3 z^6+4 x^3 z^5+x^3 z^4+2 x^2 z^{10}+3 x^2
   z^9+3 x^2 z^8 \\ 
&\quad +2 x^2 z^7+x^2 z^6 +x z^{13}+x z^{12}+2 x z^{11}+2 x z^{10}+x
   z^9+z^{17}+z^{15}+z^{14}.
\end{align*}

\begin{figure}[htbp]
\centering
\begin{tikzpicture}[scale=0.7]
\node[fill=black,circle,inner sep=1.5pt]  at (0,0) {};
\node[fill=black,circle,inner sep=1.5pt]  at (-1.5,-1) {};
\node[fill=black,circle,inner sep=1.5pt]  at (0,-1) {};
\node[fill=black,circle,inner sep=1.5pt]  at (1.5,-1) {};
\node[fill=black,circle,inner sep=1.5pt]  at (-1.5,-2) {};
\node[fill=black,circle,inner sep=1.5pt]  at (-0.5,-2) {};
\node[fill=black,circle,inner sep=1.5pt]  at (0.5,-2) {};
\node[fill=black,circle,inner sep=1.5pt]  at (-1.5,-3) {};
\node[fill=black,circle,inner sep=1.5pt]  at (-0.5,-3) {};
\node[fill=black,circle,inner sep=1.5pt]  at (-2.5,-4) {};
\node[fill=black,circle,inner sep=1.5pt]  at (-1.5,-4) {};
\node[fill=black,circle,inner sep=1.5pt]  at (-0.5,-4) {};
\node[fill=black,circle,inner sep=1.5pt]  at (-2.5,-5) {};
\node[fill=black,circle,inner sep=1.5pt]  at (-1.5,-5) {};
\node[fill=black,circle,inner sep=1.5pt] at (1.5,1) {};
\node[fill=black,circle,inner sep=1.5pt]  at (3,0) {};
\node[fill=black,circle,inner sep=1.5pt]  at (3,-1) {};
\node[fill=black,circle,inner sep=1.5pt]  at (3,-2) {};

\draw (0,0)--(-1.5,-1);
\draw (0,0)--(0,-1);
\draw (0,0)--(1.5,-1);
\draw (-1.5,-1)--(-1.5,-3);
\draw (0,-1)--(-0.5,-2);
\draw (0,-1)--(0.5,-2);
\draw (-0.5,-2)--(-0.5,-3);
\draw (-1.5,-3)--(-2.5,-4);
\draw (-2.5,-4)--(-2.5,-5);
\draw (-1.5,-3)--(-1.5,-5);
\draw (-1.5,-3)--(-0.5,-4);
\draw (1.5,1)--(0,0);
\draw (1.5,1)--(3,0);
\draw (3,0)--(3,-2);

\node at (1,-5) {$A_1$};

\node[fill=black,circle,inner sep=1.5pt]  at (7,0) {};
\node[fill=black,circle,inner sep=1.5pt]  at (6.5,-1) {};
\node[fill=black,circle,inner sep=1.5pt]  at (8,-1) {};
\node[fill=black,circle,inner sep=1.5pt]  at (6,-2) {};
\node[fill=black,circle,inner sep=1.5pt]  at (7.5,-2) {};
\node[fill=black,circle,inner sep=1.5pt]  at (5.5,-3) {};
\node[fill=black,circle,inner sep=1.5pt]  at (6.5,-3) {};
\node[fill=black,circle,inner sep=1.5pt]  at (7.5,-3) {};
\node[fill=black,circle,inner sep=1.5pt]  at (5.5,-4) {};
\node[fill=black,circle,inner sep=1.5pt]  at (7.5,-4) {};
\node[fill=black,circle,inner sep=1.5pt]  at (8.5,1) {};
\node[fill=black,circle,inner sep=1.5pt]  at (10,0) {};
\node[fill=black,circle,inner sep=1.5pt]  at (10,-1) {};
\node[fill=black,circle,inner sep=1.5pt]  at (9,-2) {};
\node[fill=black,circle,inner sep=1.5pt]  at (10,-2) {};
\node[fill=black,circle,inner sep=1.5pt]  at (11,-2) {};
\node[fill=black,circle,inner sep=1.5pt]  at (9,-3) {};
\node[fill=black,circle,inner sep=1.5pt]  at (10,-3) {};

\draw (7,0)--(8,-1);
\draw (7,0)--(5.5,-3);
\draw (5.5,-3)--(5.5,-4);
\draw (6,-2)--(6.5,-3);
\draw (6.5,-1)--(7.5,-2);
\draw (7.5,-2)--(7.5,-4);
\draw (8.5,1)--(7,0);
\draw (8.5,1)--(10,0);
\draw (10,0)--(10,-3);
\draw (10,-1)--(9,-2);
\draw (10,-1)--(11,-2);
\draw (9,-2)--(9,-3);

\node at (8,-5) {$A_2$};

\end{tikzpicture}
\caption{Trees $A_1$ and $A_2$.}\label{fig:2trees}
\end{figure}

Once a pair of nonisomorphic trees with the same associated polynomial has been found, the following corollary is an easy consequence:

\begin{cor}\label{cor:noniso}
For every positive integer $k$, there exists a $k$-tuple of nonisomorphic rooted trees $T_1,T_2,\ldots,T_k$ such that
\begin{equation}\label{eq:k-tuple}
\Pp(T_1;x,1,z) = \Pp(T_2;x,1,z) = \cdots = \Pp(T_k;x,1,z).
\end{equation}
Moreover, the proportion of rooted trees $T$ that are not uniquely determined by $\Pp(T;x,1,z)$ among all rooted trees with $n$ vertices tends to $1$ as $n \to \infty$.
\end{cor}

\begin{proof}
To prove the first part, we can for example take $T_j$ to be the tree consisting of $j-1$ copies of $A_1$ and $k-j$ copies of $A_2$, attached to a common root. The resulting trees $T_1,T_2,\ldots,T_k$ are clearly nonisomorphic, and~\eqref{eq:k-tuple} follows from~\eqref{eq:PRecursion}.

For the second part, we use a classical argument due to Schwenk \cite{schwenk} who proved the analogous statement for the spectrum (equivalently, the characteristic polynomial) of the adjacency matrix. Due to the recursion for the polynomial $\Pp$, the polynomial $\Pp(T;x,1,z)$ does not change if a fringe subtree of $T$ that is isomorphic to $A_1$ is replaced by a fringe subtree that is isomorphic to $A_2$. The proportion of rooted trees that contain $A_1$ as a fringe subtree tends to $1$ as $n \to \infty$; in fact, the number of copies of $A_1$ satisfies a central limit theorem, see e.g. \cite[Section 3]{wagneradditive}. Each such tree $T$ is not uniquely determined by $\Pp(T;x,1,z)$, since one can replace the copy of $A_1$ by $A_2$ to obtain a nonisomorphic tree with the same associated polynomial. This completes the proof.
\end{proof}

The same also applies to another natural two-variable specialisation of $\Pp(T;x,y,z)$, namely
\[\Pp(T;x,x,z) = \sum_{A \in \Aa(T)}x^{|A|}\,z^{s(A)}\]
(the sum representation follows from Theorem~\ref{thmanti}). Again, one can find two nonisomorphic trees with the same associated polynomial, see Figure~\ref{fig:2trees2}:
\[\Pp(A_3;x,x,z) = \Pp(A_4;x,x,z) = x z^8 +  x^2 z^6 + x^2 z^5 +  x^3 z^4 + 2 x^3 z^3 + x^3 z^2 + x^4 z + x^4.\]
Thus the analogue of Corollary~\ref{cor:noniso} holds as well (as it does for the univariate polynomials $\Pp(T;x,1,1)$, $\Pp(T;1,y,1)$ and $\Pp(T;1,1,z)$).

\begin{figure}[htbp]
\centering
\begin{tikzpicture}[scale=0.7]
\node[fill=black,circle,inner sep=1.5pt]  at (0,0) {};
\node[fill=black,circle,inner sep=1.5pt]  at (-1.5,-1) {};
\node[fill=black,circle,inner sep=1.5pt]  at (1.5,-1) {};
\node[fill=black,circle,inner sep=1.5pt]  at (-1.5,-2) {};
\node[fill=black,circle,inner sep=1.5pt]  at (0.5,-2) {};
\node[fill=black,circle,inner sep=1.5pt]  at (2.5,-2) {};
\node[fill=black,circle,inner sep=1.5pt]  at (2.5,-3) {};
\node[fill=black,circle,inner sep=1.5pt]  at (1.5,-4) {};
\node[fill=black,circle,inner sep=1.5pt]  at (3.5,-4) {};

\draw (0,0)--(-1.5,-1);
\draw (0,0)--(1.5,-1);
\draw (-1.5,-1)--(-1.5,-2);
\draw (1.5,-1)--(0.5,-2);
\draw (1.5,-1)--(2.5,-2);
\draw (2.5,-2)--(2.5,-3);
\draw (2.5,-3)--(1.5,-4);
\draw (2.5,-3)--(3.5,-4);

\node at (0.5,-5) {$A_3$};

\node[fill=black,circle,inner sep=1.5pt]  at (7.5,0) {};
\node[fill=black,circle,inner sep=1.5pt]  at (6,-1) {};
\node[fill=black,circle,inner sep=1.5pt]  at (9,-1) {};
\node[fill=black,circle,inner sep=1.5pt]  at (6,-2) {};
\node[fill=black,circle,inner sep=1.5pt]  at (8,-2) {};
\node[fill=black,circle,inner sep=1.5pt]  at (10,-2) {};
\node[fill=black,circle,inner sep=1.5pt]  at (5,-3) {};
\node[fill=black,circle,inner sep=1.5pt]  at (7,-3) {};
\node[fill=black,circle,inner sep=1.5pt]  at (5,-4) {};

\draw (7.5,0)--(6,-1);
\draw (7.5,0)--(9,-1);
\draw (6,-1)--(6,-2);
\draw (6,-2)--(5,-3);
\draw (6,-2)--(7,-3);
\draw (5,-3)--(5,-4);
\draw (9,-1)--(8,-2);
\draw (9,-1)--(10,-2);

\node at (7.5,-5) {$A_4$};

\end{tikzpicture}
\caption{Trees $A_3$ and $A_4$.}\label{fig:2trees2}
\end{figure}

\section{Extension to a class of posets}\label{secpos}

In this section, we consider a specific class of posets and generalise our polynomial from rooted trees to those posets.
\begin{defn}\label{defnvpo}
A poset is called a $\V$-poset if it can be generated by repeated application of the following three operations:
	\begin{enumerate}[label=(\alph*)]
	\item taking a disjoint union of smaller $\V$-posets (Figure~\ref{fig:posetconstruction1}),
	\item adding a new greatest element to a $\V$-poset (Figure~\ref{fig:posetconstruction2}, right),
	\item adding a new least element to a $\V$-poset (Figure~\ref{fig:posetconstruction2}, left),
	\end{enumerate}
starting from an empty poset. In particular, the empty set is considered a $\V$-poset for convenience.	
\end{defn}

Figure~\ref{cha.pos} shows an example of a $\V$-poset; the inspiration for the name comes from the V-shapes formed by adding a new greatest or least element, see Figure~\ref{fig:posetconstruction2}.

\begin{figure}[htbp]
\centering
\begin{tikzpicture}[scale=0.7]

\node[fill=black,circle,inner sep=1.5pt]  at (0,0) {};
\node[fill=black,circle,inner sep=1.5pt]  at (-1,1) {};
\node[fill=black,circle,inner sep=1.5pt]  at (1,1) {};
\node[fill=black,circle,inner sep=1.5pt]  at (-1,2) {};
\node[fill=black,circle,inner sep=1.5pt]  at (1,2) {};

\draw (0,0)--(-1,1)--(-1,2);
\draw (0,0)--(1,1)--(1,2);

\node[fill=black,circle,inner sep=1.5pt]  at (2.5,0) {};
\node[fill=black,circle,inner sep=1.5pt]  at (4.5,0) {};
\node[fill=black,circle,inner sep=1.5pt]  at (3.5,1) {};
\node[fill=black,circle,inner sep=1.5pt]  at (3.5,2) {};

\draw (2.5,0)--(3.5,1)--(4.5,0);
\draw (3.5,1)--(3.5,2);

\node[fill=black,circle,inner sep=1.5pt]  at (7,0) {};
\node[fill=black,circle,inner sep=1.5pt]  at (6,1) {};
\node[fill=black,circle,inner sep=1.5pt]  at (8,1) {};
\node[fill=black,circle,inner sep=1.5pt]  at (7,2) {};

\draw (7,0)--(6,1)--(7,2)--(8,1)--(7,0);

\end{tikzpicture}
\caption{Disjoint union of three posets}\label{fig:posetconstruction1}
\end{figure}

\begin{figure}[htbp]
\centering
\begin{tikzpicture}[scale=0.6]

\node[fill=black,circle,inner sep=1.5pt]  at (0,0) {};
\node[fill=black,circle,inner sep=1.5pt]  at (-1,1) {};
\node[fill=black,circle,inner sep=1.5pt]  at (1,1) {};
\node[fill=black,circle,inner sep=1.5pt]  at (-1,2) {};
\node[fill=black,circle,inner sep=1.5pt]  at (1,2) {};

\draw (0,0)--(-1,1)--(-1,2);
\draw (0,0)--(1,1)--(1,2);

\node[fill=black,circle,inner sep=1.5pt]  at (2.5,0) {};
\node[fill=black,circle,inner sep=1.5pt]  at (4.5,0) {};
\node[fill=black,circle,inner sep=1.5pt]  at (3.5,1) {};
\node[fill=black,circle,inner sep=1.5pt]  at (3.5,2) {};

\draw (2.5,0)--(3.5,1)--(4.5,0);
\draw (3.5,1)--(3.5,2);

\node[fill=black,circle,inner sep=1.5pt]  at (7,0) {};
\node[fill=black,circle,inner sep=1.5pt]  at (6,1) {};
\node[fill=black,circle,inner sep=1.5pt]  at (8,1) {};
\node[fill=black,circle,inner sep=1.5pt]  at (7,2) {};

\draw (7,0)--(6,1)--(7,2)--(8,1)--(7,0);

\node[fill=black,circle,inner sep=1.5pt]  at (3.5,-2) {};

\draw [thick] (3.5,-2)--(0,0);
\draw [thick] (3.5,-2)--(2.5,0);
\draw [thick] (3.5,-2)--(4.5,0);
\draw [thick] (3.5,-2)--(7,0);

\fill [opacity=0.2,gray] (0,0)--(3.5,-2)--(7,0)--cycle;

\node[fill=black,circle,inner sep=1.5pt]  at (12,-2) {};
\node[fill=black,circle,inner sep=1.5pt]  at (11,-1) {};
\node[fill=black,circle,inner sep=1.5pt]  at (13,-1) {};
\node[fill=black,circle,inner sep=1.5pt]  at (11,0) {};
\node[fill=black,circle,inner sep=1.5pt]  at (13,0) {};

\draw (12,-2)--(11,-1)--(11,0);
\draw (12,-2)--(13,-1)--(13,0);

\node[fill=black,circle,inner sep=1.5pt]  at (14.5,-2) {};
\node[fill=black,circle,inner sep=1.5pt]  at (16.5,-2) {};
\node[fill=black,circle,inner sep=1.5pt]  at (15.5,-1) {};
\node[fill=black,circle,inner sep=1.5pt]  at (15.5,0) {};

\draw (14.5,-2)--(15.5,-1)--(16.5,-2);
\draw (15.5,-1)--(15.5,0);

\node[fill=black,circle,inner sep=1.5pt]  at (19,-2) {};
\node[fill=black,circle,inner sep=1.5pt]  at (18,-1) {};
\node[fill=black,circle,inner sep=1.5pt]  at (20,-1) {};
\node[fill=black,circle,inner sep=1.5pt]  at (19,0) {};

\draw (19,-2)--(18,-1)--(19,0)--(20,-1)--(19,-2);

\node[fill=black,circle,inner sep=1.5pt]  at (15,2) {};

\draw [thick] (15,2)--(11,0);
\draw [thick] (15,2)--(13,0);
\draw [thick] (15,2)--(15.5,0);
\draw [thick] (15,2)--(19,0);

\fill [opacity=0.2,gray] (11,0)--(15,2)--(19,0)--cycle;

\end{tikzpicture}
\caption{Adding a new least or greatest element}\label{fig:posetconstruction2}
\end{figure}

\begin{rem}
If either the second or the third item is removed from the list of feasible operations, we obtain the family of all posets whose Hasse diagram is a union of rooted trees (the roots being either all maximal or all minimal elements).
\end{rem}

The following definition introduces the notion of basic elements, which generalise leaves of a tree (see Proposition~\ref{prop:treehasse}).
\begin{defn}\label{basic}
An element $x$ of a $\V$-poset is called \emph{basic} if it satisfies the following axioms:

\begin{enumerate}[label=B.\arabic*]
\item There are no two incomparable elements $u,v$ in the poset such that $x>u$ and $x>v$. \label{bas1}

\item There are no two incomparable elements $u,v$ in the poset such that $x<u$ and $x<v$. \label{bas2}

\item There is no element $u$ in the poset such that $u<x$ and for all $w \neq u,x$
\[u\geq w \Longleftrightarrow x \geq w\] and \[u \leq w \Longleftrightarrow x \leq w.\] \label{bas3}
\end{enumerate}

\begin{rem}
In~\ref{bas3}, one could also replace the condition $u < x$ by $u > x$ to obtain an essentially equivalent definition (the basic elements under the new condition correspond to the basic elements in the poset with all order relations reversed).
\end{rem}

Note that the only element of a single-element poset is basic. 
An element is said to be \emph{associated} with a set of basic elements $B$ if it is comparable to all the elements of $B$ and incomparable to all other basic elements. A non-basic element $u$ is called \emph{upper element} if there exists a basic element $b$ such that $u >b$. Analogously, a non-basic element $\ell$ is called \emph{lower} element if there exists a basic element $b$ such that $\ell < b$.

\begin{exa}
Let us consider the poset in Figure~\ref{cha.pos}. The basic elements in $P$ are $v_5,v_6,v_8,v_{10}$, the upper elements are $v_1,v_2,v_3,v_4,v_9$, and the only lower element is $v_7$. 
\begin{figure}[htbp]
\centering
\begin{tikzpicture}[scale=0.8]
\node[fill=black,circle,inner sep=1pt] (t1) at (0,0) {};
\node[fill=black,circle,inner sep=1pt] (t2) at (-2,-1) {};
\node[fill=black,circle,inner sep=1pt] (t3) at (-2.5,-2) {};
\node[fill=black,circle,inner sep=1pt] (t4) at (-1.5,-2) {};
\node[draw,fill=white,circle,inner sep=1pt] (t5) at (-2.5,-3) {};
\node[draw,fill=white,circle,inner sep=1pt] (t6) at (-1.5,-3) {};
\node[fill=black,circle,inner sep=1pt] (t7) at (-0.5,-4) {};
\node[draw,fill=white,circle,inner sep=1pt] (t8) at (1,-2) {};
\node[fill=black,circle,inner sep=1pt] (t9) at (3,-2) {};
\node[draw,fill=white,circle,inner sep=1pt] (t10) at (3,-3) {};
\draw (t1)--(t2)--(t3)--(t5)--(t7);
\draw (t2)--(t4)--(t6)--(t7);
\draw (t1)--(t8)--(t7);
\draw (t1)--(t9)--(t10);
\node at (0.3,0.1) {$v_1$};
\node at (-2.3,-0.8) {$v_2$};
\node at (-2.8,-2) {$v_3$};
\node at (-1.2,-2) {$v_4$};
\node at (-2.8,-3) {$v_5$};
\node at (-1.2,-3) {$v_6$};
\node at (-0.5,-4.3) {$v_7$};
\node at (1.3,-2) {$v_8$};
\node at (3.3,-2) {$v_9$};
\node at (3.3,-3.3) {$v_{10}$};
\end{tikzpicture}
\caption{A $\V$-poset $P$.}
\label{cha.pos}
\end{figure}
\end{exa}

\end{defn}

\begin{lem}\label{status}
A greatest element and a least element cannot be basic, except for linearly ordered sets, where the least element is basic.

The status of an element does not change under the three operations stated in Definition~\ref{defnvpo}, except for linearly ordered sets. Namely, a basic/non-basic/upper/lower element remains a basic/non-basic/upper/lower element.
\end{lem}

\begin{proof}
Let us consider the three operations stated in Definition~\ref{defnvpo}.

\textbf{Case 1:} Taking a disjoint union of $\V$-posets. The first statement does not apply here. Non-basic elements remain non-basic elements: if one of the three properties is violated in a component, then it is still violated after taking the union. Conversely, if one of the properties is violated for an element $x$, then it must already be violated in its component, so basic elements stay basic elements as well. The status of upper and lower elements remains for the same reason.

\textbf{Case 2:} Adding a greatest element $g$ to a $\V$-poset $P$. Since $g$ either does not satisfy~\ref{bas1} (if the poset contains incomparable elements) or does not satisfy~\ref{bas3} (otherwise), $g$ is not a basic element and the first statement holds.
Next, let $b$ be a basic element of $P$. Since $g$ is comparable to all the elements of $P$, $b$ still satisfies~\ref{bas1} and~\ref{bas2}. In addition $g$ is greater than $b$, so $b$ satisfies~\ref{bas3} as well. Thus, $b$ remains basic in the new poset. On the other hand, it is straightforward that every non-basic element $a$ remains non-basic (if one of the conditions is violated before $g$ is added, then it is still violated afterwards), thus lower and upper elements do not change their status either.

\textbf{Case 3:} Adding a least element $\ell$ to a $\V$-poset $P$. Let us first consider the case that $P$ is not linearly ordered. Then $P$ contains incomparable elements $u$ and $v$, so $\ell$ does not satisfy~\ref{bas2}. Thus, $\ell$ is not a basic element. Next, let $b$ be a basic element of $P$. With a similar reasoning as before, we see that $b$ still satisfies~\ref{bas1} and~\ref{bas2}.
The only way~\ref{bas3} could be violated through the addition of $\ell$ is that $\ell$ takes the role of $u$ in~\ref{bas3}. But since $\ell \leq w$ for all $w \in P$, this would also imply $b \leq w$ for all $w \in P$, i.e., $b$ would be the least element of $P$. Since $P$ was assumed not to be linearly ordered, there must now be two incomparable elements $u,v$ in $P$ such that $u > b$ and $v > b$, contradicting the assumption that $b$ is basic in $P$. So $b$ still satisfies~\ref{bas3} and remains basic. Again, it is clear that a non-basic element remains non-basic, and that lower and upper elements do not change their status either.

Now, for a linearly ordered set $P$, it is easy to see that the only basic element of $P$ is its least element in view of condition~\ref{bas3}. However when we add a new least element $\ell$, $\ell$ becomes the basic element, and the least element of $P$ changes its status to non-basic as it does not satisfy~\ref{bas3} anymore.
\end{proof}

\begin{rem}
When $\V$-posets are constructed by means of Definition~\ref{defnvpo}, we can assume without loss of generality that linearly ordered posets are always built by adding greatest elements. Then the exceptional case of Lemma~\ref{status} never applies.
\end{rem}

\begin{prop}\label{prop:treehasse}
For posets whose Hasse diagrams are rooted trees, the basic elements are characterised as follows.
\begin{enumerate}[label=(\alph*)]

\item If $P$ is a tree poset where the root is the greatest element, the basic elements are the leaves.

\item If $P$ is a tree poset where the root is the least element, then the basic elements are the minimal elements of the induced subposet 
\[P_1=\{x \in P \,:\, \text{there is exactly one leaf $\ell$ such that $x \leq \ell$}\}.\]
\end{enumerate}
\end{prop}

\begin{proof}
In the first case, if the root is the greatest element, then by part~\ref{bas1} in Definition~\ref{basic}, the basic elements are vertices with at most one leaf descendant. Furthermore, by~\ref{bas3}, they have to be minimal, which corresponds exactly to the leaves of the tree.

In the second case, by~\ref{bas2}, the basic elements are vertices with at most one leaf descendant, which are precisely the elements of $P_1$. By~\ref{bas3}, they have to be the minimal elements of $P_1$.
\end{proof}

We denote the family of all $\V$-posets by $\V$. Let us first construct the generalisation of our polynomial in a recursive way.

\begin{defn}\label{chat.defn2}
We define the polynomial $\Pp(P;x,y,z)$ for a poset $P \in \V$ as follows:
\begin{align*}
\Pp(\emptyset;x,y,z)&=1,\\
\Pp(\bullet;x,y,z)&=x,\\
\Pp(\cup_i P_i;x,y,z)&=\prod_{i} \Pp(P_i;x,y,z), \quad \text{where $\cup_iP_i$ is a disjoint union},\\
\Pp(P\cup \{g\};x,y,z)&=\Pp(P\cup \{\ell\};x,y,z)=\Pp(P;x,y,z)+yz^{|P|},
\end{align*}
where $g$ and $\ell$ are respectively a greatest and a least element.
\end{defn}

\begin{exa}
Let us compute the polynomial associated with the poset $P$ in Figure~\ref{cha.pos}.
\begin{align*}
\Pp(P;x,y,z)&=\big(\Pp(\{v_2,v_3,v_4,v_5,v_6\};x,y,z) \Pp(\{v_8\};x,y,z)+yz^{6} \big)\Pp(\{v_9,v_{10}\};x,y,z) + yz^{9}\\
&=\big(x((x+yz)^2+yz^4)+yz^6\big)(x+yz)+yz^9\\
&= y z^9 + y^2 z^7 + x y z^6 + x y^2 z^5 + x^2 y z^4  + x y^3 z^3 + 3 x^2 y^2 z^2 + 3 x^3 y z  + x^4.
\end{align*}
\end{exa}

As for rooted trees, we can also represent the polynomial in terms of maximal antichains.

\begin{defn}
An \emph{antichain} is a subset of a poset in which any two distinct elements are incomparable. A \emph{maximal antichain} is an antichain that is not a proper subset of any other antichain. A \emph{chain} is a linearly ordered subset of a poset $P$. A \emph{cutset} is a set of elements in a poset that intersects every maximal chain.
\end{defn}

\begin{prop}\label{proanti}
In a $\V$-poset $P$, the following holds:
\begin{itemize}
\item Every non-basic element is comparable to at least one basic element.
\item The basic elements form a maximal antichain.
\item If an element $x$ is associated with the set of all basic elements in $P$, then the only maximal antichain containing $x$ is $\{x\}$.
\end{itemize}

\end{prop}

\begin{proof}
We reason by induction on the size $n$ of the poset. If $n=1$, the statements hold trivially. Assume they hold whenever $n \leq k$, for some $k \geq 1$. In view of Definition~\ref{defnvpo}, three operations can be performed to obtain such a poset:

\textbf{Case 1:} Taking the union of disjoint $\V$-posets of size less than or equal to $k$. 
\begin{itemize}
\item By the induction hypothesis, every non-basic element from each component is comparable to at least one basic element. Moreover, by Lemma~\ref{status} the status of all elements does not change under this operation. Hence, the first statement holds.  
\item By the induction hypothesis, the basic elements of each component form a maximal antichain in their respective $\V$-posets. However, a maximal antichain of a disjoint union of posets is formed by a combination of maximal antichains in each poset. Thus, the second statement holds.
\item The last statement cannot occur here, since there is no element that is associated with the set of all basic elements (the posets are disjoint, and each of them contains at least one basic element).
\end{itemize}  

\textbf{Case 2:} Adding a greatest element $g$ to a $\V$-poset $P$ of size $k$.
\begin{itemize}
\item By Lemma~\ref{status}, the non-basic elements of the new poset are the non-basic elements of $P$, plus $g$, and the basic elements of $P$ remain basic in the new poset. Moreover, by the induction hypothesis, every non-basic element of $P$ is comparable to at least one basic element in $P$, and $g$ is comparable to all the elements of $P$, in particular the basic elements of $P$. Thus, the first statement holds.
\item By Lemma~\ref{status} again, the basic elements of $P$ are the only basic elements in $P \cup \{g\}$. By the induction hypothesis we know that these basic elements form a maximal antichain in $P$. In addition, a maximal antichain in $P$ is still a maximal antichain in $P\cup \{g\}$, since $g$ is comparable to all elements. Thus the second statement holds.
\item Let $x$ be an element associated with the set of all basic elements in $P \cup \{g\}$. Either $x$ is an element of $P$ associated with the basic elements in $P$ or $x=g$. The statement is true in the first case by the induction hypothesis and Lemma~\ref{status}. If $x=g$, $g$ is comparable to all the elements of $P$, so the only antichain containing $g$ is $\{g\}$. 
\end{itemize}

\textbf{Case 3:} Adding a least element $\ell$. Here we can use a similar argument as in Case 2. 
\end{proof}

Let $A$ be a maximal antichain in $P$. We write $b(A)$ for the number of basic elements in $A$. For $a \in A$, we define a set $P_a$ related to $a$ as follows. If $a$ is a basic element, $P_a$ is the empty set. Next, let $a$ be a non-basic element, and let $B$ be the set of basic elements to which it is comparable (so $a$ is associated with the set $B$). 

\begin{itemize}
\item If $a$ is a lower element,
\[P_a=\{b \in P:a<b, \nexists\, c \: \text{with}\: c<b \,\text{and $a$ incomparable to $c$}\}.\]
\item If $a$ is an upper element,  
\begin{align*}
P_a &=\{b \in P:a>b,\nexists\, c \; \text{with}\; c>b \,\text{and $a$ incomparable to $c$}\} \\
&\qquad \Ba \{\ell \in P: \ell \; \text{is a lower element associated with $B$}\}.
\end{align*}
\end{itemize}

Now we define
\[s(A)=\sum_{a \in A} |P_a|.\]

\begin{exa}
Let us consider again the $\V$-poset in Figure~\ref{cha.pos}. We have
\begin{align*}
P_{v_5}&=P_{v_6}=P_{v_8}=P_{v_{10}}=\emptyset,\\
P_{v_1}&=\{v_2,v_3,v_4,v_5,v_6,v_7,v_8,v_9,v_{10}\},\\
P_{v_2}&=\{v_3,v_4,v_5,v_6\},\\
P_{v_3}&=\{v_5\},\quad P_{v_4}=\{v_6\}, \quad P_{v_9}=\{v_{10}\},\\
P_{v_7}&=\{v_2,v_3,v_4,v_5,v_6,v_8\}.
\end{align*}
\end{exa}

\begin{lem}\label{lembs}
Let $P$ be a $\V$-poset and $A$ a maximal antichain of $P$. Let us add a greatest element $g$ to $P$. For $a \in A$, we denote by $P_a$ the set related to $a$ in $P$, defined as above, and $P'_a$ the set related to $a$ in $P\cup \{g\}$. The following statements hold:
\begin{enumerate}[label=(\alph*)]
\item the number of basic elements $b(A)$ in $A$ remains the same,
\item if $a$ is an upper element, then $P'_a=P_a$,
\item if $a$ is a lower element not associated with the set of all basic elements in $P$, $P'_a=P_a$,
\item if $a$ is a lower element associated with the set of all basic elements in $P$, $P'_a=P_a \cup \{g\}$. 
\end{enumerate}
\end{lem}

\begin{proof}
By Lemma~\ref{status}, the greatest element $g$ is not a basic element and the other elements do not change their status. Thus, $b(A)$ does not change.

Next, if $a$ is an upper element, then $g \not \in P'_a$ since $g >a$. All other elements of $P_a$ still satisfy the definition and remain in $P_a'$. On the other hand, no new elements satisfy the definition of $P_a$. Hence, we have $P'_a=P_a$.

If $a$ is a lower element not associated with the set of all basic elements in $P$, then there exists a basic element $c$ such that $g > c$ and $c$ is incomparable to $a$. So $g \not \in P'_a$, and again $P'_a=P_a$.

If $a$ is a lower element associated with the set of all basic elements in $P$, then $g$ satisfies the definition of $P_a'$ while all elements of $P_a$ remain. Thus $P'_a=P_a \cup \{g\}$.
\end{proof}

The same reasoning can be adapted to adding a least element, where we obtain the following lemma, which is in fact slightly simpler.

\begin{lem}\label{lemleast}
Let $P$ be a $\V$-poset that is not linearly ordered and $A$ a maximal antichain of $P$. Let us add a least element $\ell$ to $P$. For $a \in A$, we denote by $P_a$ the set related to $a$ in $P$, defined as above, and $P'_a$ the set related to $a$ in $P\cup \{\ell\}$. The following statements hold:
\begin{enumerate}[label=(\alph*)]
\item the number of basic elements $b(A)$ in $A$ remains the same,
\item for every element $a \in A$, $P'_a=P_a$.
\end{enumerate}
\end{lem}

\begin{thm}
Let $\Aa(P)$ be the set of all maximal antichains of a nonempty poset $P \in \V$. We have 
\[\Pp(P;x,y,z)= \sum_{A \in \Aa(P)}x^{b(A)}\,y^{|A|-b(A)}z^{s(A)}.\]
\end{thm}

\begin{proof}
Let $R(P;x,y,z)$ denote the right side of the equation. We use induction on the size of $P$. Since the only element of the single-element poset $\bullet$ is basic, $R(\bullet;x,y,z)=x=\Pp(\bullet;x,y,z)$ trivially holds.
Now for $P \neq \bullet$, we have to check that $R(P;x,y,z)$ satisfies Definition~\ref{chat.defn2}.

Suppose first that $P$ is a disjoint union of $P_1, \dots, P_k$. The maximal antichains of $P$ are precisely the unions of maximal antichains in $P_1,\dots, P_k$:
\[\Aa(P)=\{A_1\cup A_2\cup \dots \cup A_k \,:\, A_i \in \Aa(P_i) \, \text{for}\, 1\leq i \leq k\}.\]
Note that $b(A)=\sum_{i=1}^k b(A_i)$ and $s(A)=\sum_{i=1}^k s(A_i)$ if $A$ is the union of antichains $A_1,A_2,\ldots,A_k$ in $P_1,P_2,\ldots,P_k$ respectively.
Thus,
\begin{align*}
R(P;x,y,z) &=\sum_{A \in \Aa(P)}x^{b(A)}\,y^{|A|-b(A)}\,z^{s(A)} \\
&=\prod_{i=1}^k\Big(\sum_{A_i \in \Aa(P_i)} x^{b(A_i)}y^{|A_i|-b(A_i)}\,z^{s(A_i)}\Big) \\
&=\prod_{i=1}^k R(P_i;x,y,z).
\end{align*} 

Now we consider the scenario that a new greatest element is added to a poset $P$: $P' = P \cup \{g\}$. The maximal antichains of $P'$ are precisely the maximal antichains of $P$, plus the antichain $\{g\}$.

Let $A \in \Aa(P)$. Suppose first that $A$ does not contain a lower element associated with all the basic elements in $P$. By Lemma~\ref{lembs}, adding a greatest element $g$ to $P$ does not affect $b(A)$ or $s(A)$. 

Otherwise, $A$ contains a lower element $x$ associated with all the basic elements in $P$. By Proposition~\ref{proanti}, $A=\{x\}$. The last item of Proposition~\ref{proanti} also implies that any two lower elements that are associated with all the basic elements must be comparable. So let $\ell_1,\ell_2\dots,\ell_s$ be the lower elements associated with all the basic elements in $P$ such that $\ell_1>\ell_2>\cdots>\ell_s$. By Lemma~\ref{lembs}, adding a greatest element $g$ will change the $P_{\ell_i}$s as follows: $P_{\ell_i}'=P_{\ell_i} \cup \{g\}$, so $|P'_{\ell_i}|=|P_{\ell_{i+1}}|$ for $i=1,\dots,s-1$. And according to the definition, $|P'_g|=|P|-|\{\ell_1,\dots \ell_s\}|=|P_{\ell_1}|$ and $|P_{\ell_s}'|=|P|$. Furthermore, the $\ell_i$s and $g$ are non-basic elements, so $b(\{\ell_i\})=b(\{g\})=0$. 

Thus, we get  

\begin{align*}
R(P\cup \{g\};x,y,z)&=\sum_{A \in \Aa(P\cup \{g\})}x^{b(A)}\,y^{|A|-b(A)}\,z^{s(A)}\\
&=\sum_{A \in \Aa(P)\Ba \{\{\ell_1\},\dots,\{\ell_s\}\}} x^{b(A)}\,y^{|A|-b(A)}\,z^{s(A)}+\sum_{i=1}^{s-1} yz^{|P'_{\ell_i}|}+yz^{|P'_g|}+yz^{|P'_{\ell_s}|}\\
&=\sum_{A \in \Aa(P)\Ba \{\{\ell_1\},\dots,\{\ell_s\}\}} x^{b(A)}\,y^{|A|-b(A)}\,z^{s(A)}+\sum_{i=1}^{s-1} yz^{|P_{\ell_{i+1}}|}+yz^{|P_{\ell_1}|}+yz^{|P|}\\
&=\sum_{A \in \Aa(P)} x^{b(A)}\,y^{|A|-b(A)}\,z^{s(A)} + yz^{|P|},
\end{align*} 
which completes the induction in this case. Following Lemma~\ref{lemleast}, a similar reasoning can be applied to adding a least element.
\end{proof}

\begin{exa}
Let us again consider the $\V$-poset in Figure~\ref{cha.pos}. The maximal antichains of $P$ are 
\begin{align*}
&\{v_1\},\{v_7,v_9\},\{v_7,v_{10}\},\{v_2,v_8,v_9\},\{v_2,v_8,v_{10}\},\\ &\{v_3,v_4,v_8,v_9\},\{v_3,v_4,v_8,v_{10}\},\{v_3,v_6,v_8,v_9\},\{v_3,v_6,v_8,v_{10}\},\\&\{v_4,v_5,v_8,v_9\},\{v_4,v_5,v_8,v_{10}\},\{v_5,v_6,v_8,v_9\},\{v_5,v_6,v_8,v_{10}\}.
\end{align*}

 Since the basic elements are $v_5,v_6,v_8,v_{10}$, these maximal antichains correspond respectively to the following monomials:
\begin{align*}
y z^9,y^2 z^7,x y z^6,x y^2 z^5,x^2 y z^4,x y^3 z^3,x^2 y^2 z^2,x^2 y^2 z^2,x^3 y z,x^2 y^2 z^2,x^3 y z,x^3 y z,x^4.
\end{align*} 
Thus,
\[\Pp(P;x,y,z)= y z^9 + y^2 z^7 + x y z^6 + x y^2 z^5 + x^2 y z^4  + x y^3 z^3 + 3 x^2 y^2 z^2 + 3 x^3 y z  + x^4,\]
as determined earlier.
\end{exa}

The special values of Corollary~\ref{cor:yz0}, Corollary~\ref{cor:x0} and Theorem~\ref{thmeval} still have essentially the same interpretations as in the special case of rooted trees, with the same proofs.

\begin{cor}
For every $\V$-poset $P$, we have
\[\Pp(P;x,y,0) = \Pp(P;x,0,z) = x^{\text{number of basic elements of $P$}}.\]
\end{cor}

\begin{cor}
Let $P$ be a $\V$-poset, and $\Aa^*(P)$ the set of maximal antichains of $P$ that do not contain a basic element. We have
\[\Pp(P;0,y,z) = \sum_{A \in \Aa^*(P)}y^{|A|}\,z^{s(A)}.\]
In particular, $\Pp(P;0,1,1)$ is the number of maximal antichains of $P$ that do not contain a basic element.
\end{cor}

\begin{thm}\label{thmeval_poset}
Let $P$ be a $\V$-poset. We have the following special values:
\begin{align*}
\Pp(P;1,1,1) &= \text{number of maximal antichains of $P$},\\
\Pp(P;2,1,1) &=\text{number of antichains of $P$ (including $\emptyset$)},\\
\Pp(P;1,1,2) &=\text{number of cutsets of $P$},\\
\Pp(P;2,1,2) &=2^{|P|}.
\end{align*}
\end{thm}

\begin{rem}
The number of antichains is also equal to the number of upward closed sets and the number of downward closed sets, cf. Remark~\ref{remideal}.
\end{rem}

\begin{rem}
The final identity generalises to
\[\Pp(P;1+y,y,1+y) =(1+y)^{|P|}\]
once again.
\end{rem}

\begin{rem}
In a $\V$-poset, maximal antichains are minimal cutsets (minimal with respect to inclusion), and vice versa. Thus $\Pp(P;1,1,1)$ also counts minimal cutsets. This is easy to verify by induction, based on the recursive definition of posets: maximal antichains, as well as minimal cutsets, in a union of disjoint posets are simply unions of maximal antichains (minimal cutsets, respectively) in the different components. Likewise, when a greatest element $g$ or a least element $\ell$ is added, all maximal antichains and minimal cutsets retain their status, while the only additional maximal antichain, which is also the only additional minimal cutset, is $\{g\}$ ($\{\ell\}$, respectively). However, in general posets, the two notions are not always equivalent: for instance, in the N-poset shown in Figure~\ref{fig.Nbow} (left), $\{v,w\}$ is a maximal antichain, but not a cutset.
\end{rem}

Let us finally remark that Proposition~\ref{prop:probability1} remains valid for posets:

\begin{prop}
Let $P$ be a $\V$-poset, and $p \in (0,1)$ be a fixed probability. If each element of $P$ is selected independently with probability $p$, the probability that the selected elements form a cutset is
\[(1-p)^{|P|} \Pp \Big(P; \frac{p}{1-p}, \frac{p}{1-p}, \frac{1}{1-p} \Big).\]
\end{prop}

Interestingly, the family of $\V$-posets also played a role in the recent paper~\cite{hasebe} by Hasebe and Tsujie in a somewhat different context. It can also be defined without using recursion, as shown in~\cite[Theorem 4.3]{hasebe}. The equivalent definition given by Hasebe and Tsujie is stated in the following proposition. Since it is quite important for the discussion below, let us also provide a proof that the two definitions are equivalent for the sake of completeness. Another characterisation of $\V$-posets was recently given in \cite{cooper}.

\begin{prop}\label{defnvpo2}
A poset $P$ is a $\V$-poset if and only if it does not contain either of the posets in Figure~\ref{fig.Nbow} as an induced subposet.
Equivalently: there are no four elements $u,v,w,x$ such that $u>w,u>x,v>x$; $u,v$ are incomparable; and $w,x$ are incomparable.
\end{prop}

\begin{figure}[htbp]
\centering
\begin{tikzpicture}[scale=0.8]
\node[fill=black,circle,inner sep=1pt] (t1) at (0,0) {};
\node[fill=black,circle,inner sep=1pt] (t2) at (2,0) {};
\node[fill=black,circle,inner sep=1pt] (t3) at (0,-1) {};
\node[fill=black,circle,inner sep=1pt] (t4) at (2,-1) {};
\node[fill=black,circle,inner sep=1pt] (t5) at (5,0) {};
\node[fill=black,circle,inner sep=1pt] (t6) at (7,0) {};
\node[fill=black,circle,inner sep=1pt] (t7) at (5,-1) {};
\node[fill=black,circle,inner sep=1pt] (t8) at (7,-1) {};
\draw (t3)--(t1)--(t4)--(t2);
\draw (t5)--(t8)--(t6)--(t7)--(t5);
\node at (0.2,0.2) {$u$};
\node at (2.2,0.2) {$v$};
\node at (0.2,-1.2) {$w$};
\node at (2.2,-1.2) {$x$};
\node at (5.2,0.2) {$u$};
\node at (7.2,0.2) {$v$};
\node at (5.2,-1.2) {$w$};
\node at (7.2,-1.2) {$x$};
\node at (1,-2) {N-poset};
\node at (6,-2) {bowtie};
\end{tikzpicture}
\caption{N-poset and bowtie.}
\label{fig.Nbow}
\end{figure}

\begin{proof}
We first show that a $\V$-poset does not contain an N-poset or a bowtie. We reason by induction on the size $n$ of the poset. If $n=1$, the claim holds trivially. Assume that it holds whenever $n \leq k$, for some $k\geq 1$. Now, consider the case where $n=k+1$. By Definition~\ref{defnvpo}, three operations can be made to obtain such a poset:

\begin{itemize}
\item Taking the union of disjoint $\V$-posets of size less than or equal to $k$. By the induction hypothesis, those posets are N- and bowtie-free, and so is their disjoint union. 

\item Adding a greatest element to a $\V$-poset $P$ of size $k$. By the induction hypothesis again, the poset $P$ is N- and bowtie-free. Moreover, since $g$ is comparable to all other elements in the poset, there is no induced N-poset nor a bowtie that has $g$ as an element. Thus, the new poset is N- and bowtie-free.

\item The same reasoning applies when a least element is added.
\end{itemize}

Now, let us prove that the converse is true, i.e., an N- and bowtie-free poset is indeed a $\V$-poset. Once again, we use induction on the size $n$ of the poset. The claim is straightforward when $n=1$. For the induction step, let us consider all maximal and minimal elements of the poset. If there exists a component with at least two maximal and at least two minimal elements, an induced N-poset or bowtie will occur. So each component has either only one maximal element or only one minimal element, which implies that each component has a least or a greatest element. Now, let us remove these elements. The remaining poset is still N- and bowtie-free, so by the induction hypothesis, the poset can be constructed as in Definition~\ref{defnvpo}.  
\end{proof}

\begin{rem}
The family of N-free (or series-parallel) posets is well-studied in the literature, see~\cite{habib} for an early study of these posets, and~\cite{bouchemakh,felsner,zaguia} for some more recent examples. In \cite{gibson}, it is shown that N-free posets are \emph{skeletal}. One of several equivalent definitions of such a poset is that every cutset contains a maximal antichain.
\end{rem}

Let us finally show that $\V$-posets are indeed the natural class on which to define our polynomial by proving that there is no polynomial with nonnegative integer coefficients that satisfies Theorem~\ref{thmeval_poset} for an N-poset or a bowtie. 

Let us first consider a bowtie. Since it has two maximal antichains, following Theorem~\ref{thmeval_poset}, the polynomial should be of the form $x^{a_1}y^{b_1}z^{c_1}+x^{a_2}y^{b_2}z^{c_2}$ for some nonnegative integers $a_1,b_1,c_1,a_2,b_2,c_2$. Moreover, the number of antichains is $7$, and by the same proposition the polynomial should satisfy $2^{a_1}+2^{a_2}=7$. This equation has no integer solutions for $a_1$ and $a_2$, so there is no polynomial with nonnegative integer coefficients which satisfies Theorem~\ref{thmeval_poset} for a bowtie.

Next, for an N-poset, the number of maximal antichains is $3$, so the polynomial should be of the form $x^{a_1}y^{b_1}z^{c_1}+x^{a_2}y^{b_2}z^{c_2}+x^{a_3}y^{b_3}z^{c_3}$ for nonnegative integers $a_1,b_1,c_1,a_2,b_2,c_2,a_3,b_3,c_3$. Moreover, the number of antichains and the number of cutsets is $8$, and $2^4=16$. Therefore, to satisfy Theorem~\ref{thmeval_poset}, we have the following system of equations:
\[\begin{cases}
8&=2^{a_1}+2^{a_2}+2^{a_3},\\
8&=2^{c_1}+2^{c_2}+2^{c_3},\\
16&=2^{a_1+c_1}+2^{a_2+c_2}+2^{a_3+c_3}.
\end{cases}
\]

For the first equation, the values $a_1,a_2,a_3$ have to be a permutation of $(2,1,1)$, so $a_1+a_2+a_3=4$. Likewise, $c_1,c_2,c_3$ have to form a permutation of $(2,1,1)$, so the sum is $c_1+c_2+c_3=4$. These two equations give us $a_1+a_2+a_3+c_1+c_2+c_3=8$. However, the exponents in the last equation have to be a permutation of $(3,2,2)$, so the sum would be $a_1+a_2+a_3+c_1+c_2+c_3=7 \neq 8$. Therefore there is no polynomial which satisfies Theorem~\ref{thmeval_poset} for an N-poset either.

\bibliographystyle{abbrv} 
\bibliography{Polynomial}

\end{document}